\documentclass[a4paper,12pt,twoside,leqno]{article}

\title{The variable exponent BV-Sobolev capacity}
\author{Heikki Hakkarainen\footnote{Email: heikki.hakkarainen@oulu.fi} \and Matti Nuortio\footnote{Corresponding author. Email: mnuortio@paju.oulu.fi}}

\usepackage[T1]{fontenc}
\usepackage[utf8]{inputenc} 
\usepackage[british]{babel}

\usepackage{fancyhdr}

\usepackage{amsmath}
\usepackage{amssymb}
\usepackage{amsfonts}
\usepackage{latexsym}
\usepackage{amsthm}
\usepackage{enumerate}
\usepackage{booktabs}
\usepackage{bbm}

\theoremstyle{plain}
\newtheorem{teoreema}{Theorem}[section]
\newtheorem{lemma}[teoreema]{Lemma}
\theoremstyle{definition}
\newtheorem{maaritelma}[teoreema]{Definition}
\theoremstyle{remark}
\newtheorem*{huomautus}{Remark}

\newcommand{\R}{\mathbb{R}}
\newcommand{\Rn}{\mathbb{R}^{n}}
\newcommand{\Ly}{L^{1}}
\newcommand{\Lp}{L^{p(\cdot)}}
\newcommand{\W}{W^{1,p(\cdot)}}
\newcommand{\BV}{\mathrm{BV}}
\newcommand{\BVp}{\mathrm{BV}^{p(\cdot)}}
\newcommand{\BVpp}{\widetilde{\mathrm{BV}}\vphantom{\mathrm{BV}}^{p(\cdot)}}
\newcommand{\Lip}{\mathrm{Lip}_{\text{loc}}}

\newcommand{\ep}{\varepsilon}
\renewcommand{\epsilon}{\varepsilon}
\renewcommand{\emptyset}{\varnothing}

\newcommand{\modup}{\varrho_{p(\cdot)}}
\newcommand{\moduyp}{\varrho_{1,p(\cdot)}}
\newcommand{\modubvp}{\varrho_{\BVp}}
\newcommand{\modubvpomega}{\varrho_{\BVp (\Omega)}}
\newcommand{\modubvpp}{\varrho_{\BVpp}}

\newcommand{\norm}[1]{\lVert#1\rVert}

\newcommand{\kapp}{C_{p(\cdot)}}
\newcommand{\kapbvp}{C_{\BVpp}}

\newcommand{\adm}{\mathcal{A}_{\BVpp}}

\DeclareMathOperator*{\esssup}{\text{ess sup}}
\DeclareMathOperator*{\essinf}{\text{ess inf}}

\begin{document}

\maketitle

\begin{abstract}
In this article we study basic properties of the mixed BV-Sobolev capacity
with variable exponent $p$. We give an alternative way to define mixed type
BV-Sobolev-space which was originally introduced
by Harjulehto, H{\" a}st{\" o}, and Latvala. Our definition is based on relaxing the $p$-energy
functional with respect to the Lebesgue space topology. We prove that this
procedure produces a Banach space that coincides with the space defined by
Harjulehto et al.~for bounded domain $\Omega$ and
log-H{\" o}lder continuous exponent $p$. Then we show that this
induces a type of variable exponent BV-capacity and that this is a Choquet
capacity with many usual properties. Finally, we prove that
this capacity has the same null sets as the variable exponent Sobolev
capacity when $p$ is log-H{\" o}lder continuous.
\end{abstract}

\textbf{Keywords:} capacity, functions of bounded variation, Sobolev spaces, variable exponent

\textbf{Mathematics subject classification 2000:} 46E35, 26A45, 28A12


\section{Introduction}
Variable exponent analysis has become a growing field of interest during
the past 10--20 years. Variable exponent problems originated with the study
of variable exponent function spaces and certain variational problems
\cite{KovR91,Sar79,Tse61,Zhi87}.
Then the study spread out to e.g.~harmonic analysis, geometric analysis,
and fuller theory of partial differential equations.

As an introduction to the subject of variable exponent problems,
the reader is advised to the original article \cite{KovR91},
the forthcoming monograph \cite{DieHHR_pp10}, and the survey
articles \cite{DieHN04,HarHLN10,Sam09}.

In this article, we study a class of ''bounded-variation-like'' functions
and a capacity in the variable exponent setting. In general, the study of
capacity is closely related to the potential theory, say the Wiener regularity
of boundary points with respect to a variational problem. For such study
of capacity in the variable exponent case, see \cite{AlkK04}. As another example
of related variable exponent potential theory, we mention \cite{HarL08}.
Capacity is also the correct way of measuring the fine regularity properties
of Sobolev functions. For the variable exponent case see \cite[Section 5]{HarHKV03}
and also \cite{HarH04,HarHL04,HarKT07}.

Let $p$ be a finite variable exponent. The $p(\cdot)$-capacity of set $E\subset\Rn$ is defined as
\[C_{p(\cdot)}(E)=\inf\int_{\Rn} |u|^{p(x)}+|\nabla u|^{p(x)}dx,\]
where the infimum is taken over admissible functions $u\in S_{p(\cdot)}(E)$ where
    \[
        S_{p(\cdot)}(E) =
        \left\{ u\in W^{1,p(\cdot)}(\Rn) \; : \; u \geq 1 \text{ in an open set containing } E \right\} .
    \]
It is easy to see that if we restrict these admissible functions $S_{p(\cdot)}$
to the case $0 \leq u \leq 1$, we get the same capacity. In this case it is
obviously possible to also drop the absolute value from $|u|^{p(x)}$.

The $p(\cdot)$-capacity enjoys the usual desired properties of capacity
when $p^{-}>1$, see \cite{HarHKV03}. However, just as in the constant exponent
setting, some of these properties require different argument when $p^{-}=1$.
In the classical treatment of $1$-capacity, see \cite{FedZ72}, properties
such as limit property of capacity with respect to increasing sequence of sets are
first proved for BV-capacity. Then the corresponding result is obtained
for $1$-capacity by proving that these two capacities are in fact equal.
This BV-approach has been used to study questions related to $1$-capacity
for example in \cite{FedZ72} in the Euclidean setting and
in \cite{HakK10,KinKST08} in the setting of metric measure spaces. We
note that BV-capacity has also been studied without comparison to
the $1$-capacity, for general reference see \cite[Chapter 5.11]{Zie89}
in the case of BV-capacity and \cite[Chapter 4.7]{EvaG92} in the case
of $1$-capacity.

When studying the $1$-capacity, one encounters some difficulties. It has turned
out that the function space $W^{1,1}(\Omega)$ does not offer the best possible
framework for studying variational problems; instead, several difficulties arise.
On the other hand, the space of BV functions provides a better setting, bypassing
some of these difficulties. Aware of the obstacles in the constant
exponent case,
we choose this BV-based approach as our starting point. We give an alternative
definition for the mixed BV-Sobolev space of functions presented
in \cite{HarHL08}. The original definition is based on directly separating
the modular into a ''Sobolev part'' and a ''BV part'', defined on separate
parts of the domain, whereas our definition is based on relaxing
the $p(\cdot)$-energy functional over whole of the domain. The rough idea is
that the functions should behave like Sobolev functions when $p(\cdot)>1$ and
like BV-functions when $p(\cdot)=1$. For general introduction into the
procedure of relaxation of a functional,
see \cite[Chapter 1.3 and Example 1.4.2]{But89}.
See also \cite[Examples 3.13, 3.14]{Dal93}.

We obtain several properties for our mixed BV-Sobolev space of functions. We show that it has a naturally induced modular, that it is a Banach space, that the modular has an important semicontinuity property, and that the space has certain well-behaved closure properties. We also show that our definition of the class of functions coincides with \cite{HarHL08} under certain assumptions. The new definition should also work well for unbounded domains.

We conclude by defining a capacity based on the mixed BV-Sobolev space. We show that this capacity has many properties we would usually expect from a capacity related to potential theory: it is continuous with respect to an increasing sequence of sets, it defines an outer measure, and so forth. As a result, this capacity is a Choquet capacity. We finally show an equivalence between the mixed capacity and the Sobolev capacity with respect to null sets.

We note that our approach has some advantages over \cite{HarHKV03,HarHL08}.
Our mixed modular has a lower semicontinuity property which is mainly
due to the definition via relaxation. It is known that lower semicontinuity
of the modular can be used to prove many properties
of the capacity, confer \cite{FedZ72,HakK10} and \cite[Chapter 5.12]{Zie89}.
The lower semicontinuity is not known in the case of \cite{HarHKV03} and thus
properties of a similarly induced capacity remain unknown to us.

Also, in the paper \cite{HarHL08} the authors study the case of bounded
$\Omega$ and strongly log-H{\" o}lder continuous $p$. We will be able to
establish an equivalence between two definitions even after relaxing the
condition of strong log-H{\" o}lder continuity to regular log-H{\" o}lder
continuity. However, it seems to us that in \cite{HarHL08}, it is required
that $\Omega$ has finite measure. In contrast, our definition of a
mixed space does not depend on this assumption.

Finally, in \cite{HarHKV03} the Choquet property
of the variable exponent Sobolev capacity is established in the case
$p^{-} > 1$. It is not known whether it is true for $p^{-} = 1$. In
contrast, the proofs for our mixed capacity do not distinguish between
the cases $p^{-} > 1$ and $p^{-} = 1$. Our mixed capacity is a Choquet
capacity with the same null sets as the Sobolev capacity.

\section{Preliminaries}
Let $\Omega\subset\Rn$ be an open set. A measurable function $p:\Omega\to\left[1,\infty\right)$ is called a \emph{variable exponent}. Note that we may later on impose additional restrictions on the variable exponent. We denote
    \[
        p^{+} := \esssup\limits_{x\in \Omega} p(x) , \quad p^{-}=\essinf\limits_{x\in \Omega}p(x) ,
    \]
and for $E \subset \Rn$,
    \[
        p^{+}_{E} := \esssup\limits_{x\in E} p(x),\quad p^{-}_{E}= \essinf\limits_{x\in E}p(x).
    \]
The set of points where $p$ attains value $1$ will be important,
so we reserve special notation for it. Following \cite{HarHL08},
we denote
    \[
        Y := \left\{ x \in \Omega \; : \; p(x)=1 \right\} \, .
    \]

In this paper, we always assume that $p^{+}<\infty$. This assumption is typical, since it ensures that the notion of \emph{convergence in modular} is equivalent to the typical convergence in norm; we shall use this fact later on. See \cite[(2.28) on p.~598]{KovR91}. Also, we note that the concept of $\infty$-capacity is, in general, not very useful, so it is reasonable to restrict the consideration to the strictly finite case.

We define a \emph{modular} by setting
\[ \modup (u)=\int_{\Omega}|u(x)|^{p(x)} \, dx. \]
The  \emph{variable exponent Lebesgue space} $L^{p(\cdot)}(\Omega)$ consists of all measurable functions $u:\Omega\to\R$ for which the modular
$\modup(u/\lambda)$ is finite for some $\lambda>0$. We define a norm on this space as a Luxemburg norm:
\[\|u\|_{L^{p(\cdot)}(\Omega)}=\inf\left\{\lambda>0:\modup(u/\lambda)\leq 1\right\}.\] 
It is known that $\Lp(\Omega)$ is a Banach space. The variable exponent Lebesgue space is a special case of a \emph{Musielak--Orlicz space}, but here we only consider the Lebesgue and Sobolev type spaces. For constant function $p$ the
variable exponent Lebesgue space coincides with the standard Lebesgue space.

The \emph{variable exponent Sobolev space} $\W(\Omega)$ consists of functions $u\in\Lp(\Omega)$ whose distributional gradient $\nabla u$ has modulus in
$\Lp(\Omega)$. The variable exponent Sobolev space $\W(\Omega)$ is a Banach space with the norm
\[\|u\|_{1,p(\cdot)}=\|u\|_{p(\cdot)}+\|\nabla u\|_{p(\cdot)}.\]
We also define
\[\moduyp(u)=\modup(u)+\modup(\nabla u).\]

We recall the definition of log-H{\" o}lder continuity.

\begin{maaritelma}Function $p:\Omega\to\R$ is locally log-H{\" o}lder continuous on $\Omega$ if there exists $c_{1}>0$ such that
\[|p(x)-p(y)|\leq\frac{c_{1}}{\text{log}\left(e+\frac{1}{|x-y|}\right)}\]
for all $x,y\in\Omega$. We say that $p$ is globally log-H{\" o}lder continuous on $\Omega$ if it is locally log-H{\" o}lder continuous on $\Omega$
and there exists $p_{\infty} \geq 1$ and a constant $c_{2}>0$ such that
\[|p(x)-p_{\infty}|\leq\frac{c_{2}}{\text{log}(e+|x|)}\]
for all $x\in\Omega$. The constant $\max\{c_{1},c_{2}\} =: c$ is called the log-H{\" o}lder constant of $p$.
\end{maaritelma}

\begin{huomautus}
    We usually replace the constants $c_1, c_2$ by the maximum $c$. This
    is due to the fact that we may extend a log-H{\" o}lder continuous
    function to a larger domain, but in such procedure one of the constants
    may become larger. However, the maximum $c$ remains in extension.
\end{huomautus}

\begin{huomautus}
    In what follows, we usually only speak about log-H{\" o}lder continuity.
    The meaning will be clear from the context.
    In a bounded set, we mean by this local log-H{\" o}lder continuity. In
    an unbounded set, we mean by this global log-H{\" o}lder continuity.
\end{huomautus}

The assumption of log-H{\" o}lder continuity is typical in the variable exponent setting. It ensures the following important estimate:
    \[
        R^{-(p^{+}_B - p^{-}_B)} \leq C .
    \]
This for a ball $B$ of radius $R > 0$ and a uniform constant on the right hand side. We shall explicitly make use of this estimate. In general, this estimate has some important consequences, such as the density of smooth functions and that convolution-based mollifiers are available as smoothing operators. For a discussion, see the introduction to \cite{Ler05}. Assumption of log-H{\" o}lder continuity is also crucial in the regularity theory of partial differential equations with variable exponent \cite{Zhi97}.

In \cite{HarHL08}, the authors introduce a mixed BV-Sobolev-type space of functions. One of their main results is concerned with the problem of energy minimization. The authors use a slightly stronger condition for the exponent, the \emph{strong log-H{\" o}lder continuity}.
\begin{maaritelma}Exponent $p$ satisfies the strong log-H{\" o}lder continuity condition if $p$ is log-H{\" o}lder continuous in $\Omega$ and
\[\lim\limits_{x\to y}|p(x)-1|\text{log}\frac{1}{|x-y|}=0\]
for every $y \in Y$.
\end{maaritelma}
This condition is necessary for some results in the theory of minimizers and of partial differential equations. Earlier on, it was used by Acerbi and Mingione in e.g.~\cite{AceM01}.
In \cite{HarHL08}, the effect of strong log-H{\" o}lder continuity is as follows. The authors use the mollifiers as smoothing operation and show that their definition of a mixed pseudo-modular is upper semicontinuous with respect to these mollifiers and with respect to a closed subset. We repeat this as Theorem \ref{HarHL08_teor46} later on. We will also relax this result in Theorem \ref{variant_of_HarHL08_46} at the cost of a multiplicative constant.

\section{The mixed BV-Sobolev space: known results}
In order to define the mixed BV-Sobolev space, we first recall the
ordinary BV-space, i.e.~functions of bounded variation.

\begin{maaritelma}
    Denote
    \[
        \|Du\|(\Omega) :=
        \sup \left\{ \int_{\Omega} u \, \text{div} \, \varphi \, dx \; : \; \varphi\in C_{0}^{1}(\Omega;\Rn) , \, |\varphi| \leq 1 \right\} .
    \]
    A function $u\in\Ly(\Omega)$ has \emph{bounded variation}
    in $\Omega$, denote $u\in\BV(\Omega)$, if $\|Du\| (\Omega) < \infty$.    
    We denote $u\in\BV_{\text{loc}}(\Omega)$, if $u\in\BV(U)$
    for every open set $U\subset\subset\Omega$.
\end{maaritelma}

If $u\in\BV(\Omega)$, then the distributional gradient $Du$ is a vector
valued signed Radon measure and $\|Du\|(\cdot)$ is the total variation measure
associated with $Du$. A set $E\subset\Rn$ has \emph{finite perimeter}
in $\Omega$, if $\chi_{E}\in\BV(\Omega)$. The perimeter of $E$
in $\Omega$ is defined as 
    \[
        P(E,\Omega)=\|D\chi_{E}\|(\Omega).
    \]
For the properties of $\BV$-functions, e.g.~the lower semicontinuity of
total variation measure, approximation by smooth functions, and the coarea
formula, we refer to \cite[Chapter 5]{EvaG92}, \cite[Chapter 5]{Zie89},
and \cite[Chapter 1]{Giu84}.

We now present the definition of the mixed BV-Sobolev space introduced in \cite{HarHL08}. Let $\Omega$ be open and bounded and $E \subset \Omega$ Borel.

\begin{maaritelma} \label{HarHL08_maar42}
    Define the mixed-type pseudo-modular
    \[
        \varrho_{\BVp(E)}(u) := \norm{Du}(E \cap Y) + \varrho_{\Lp(E \setminus Y)}(\nabla u) .
    \]
    Define the mixed-type norm
    \[
        \norm{u}_{\BVp(\Omega)}
        := \norm{u}_{\Lp(\Omega)}
           + \inf \left\{ \lambda > 0 \; : \; \modubvpomega(u/\lambda) \leq 1 \right\} .
    \]
    Define the space $\BVp(\Omega)$ to consist of all measurable functions \linebreak
    $u : \Omega \to \R$ with $\norm{u}_{\BVp(\Omega)} < \infty$. Define also
    $u \in \BVp_{\text{loc}}(\Omega)$, if $u \in \BVp(U)$ for every open
    $U \subset \subset \Omega$.
\end{maaritelma}

By \cite[Proposition 4.3]{HarHL08}, the space $\BVp(\Omega)$ is a Banach space.

We denote the standard mollification $\varphi_\delta \ast u =: u_\delta$. The following result in \cite{HarHL08} links approximation and upper semicontinuity of the BV-Sobolev pseudo-modular.

\begin{teoreema}[Theorem 4.6 in \cite{HarHL08}] \label{HarHL08_teor46}
    Let $\Omega \subset \Rn$ be bounded and let $p$ be a bounded, strongly log-H{\" o}lder continuous variable exponent in $\Omega$. If $u \in \BVp(\Omega)$ and $F \subset \Omega$ is closed, then
    \[
        \limsup\limits_{\delta \to 0} \varrho_{\BVp(F)}(u_\delta)
        \leq \varrho_{\BVp(F)}(u) .
    \] 
\end{teoreema}

We note that if the proof of \cite[Theorem 4.6]{HarHL08} is examined
carefully, we may also state the following.

\begin{teoreema} \label{variant_of_HarHL08_46}
    Let $\Omega \subset \Rn$ be bounded and let $p$ be a bounded,
    log-H{\" o}lder continuous variable exponent in $\Omega$.
    If $u \in \BVp(\Omega)$ and $F \subset \Omega$ is closed, then
    \[
        \limsup\limits_{\delta \to 0} \varrho_{\BVp(F)}(u_\delta)
        \leq \; C \varrho_{\BVp(F)}(u)
    \]
    with $1 \leq C < \infty$.
\end{teoreema}

\begin{huomautus}
    Note that we now relaxed the condition of strong log-H{\" o}lder
    continuity to log-H{\" o}lder continuity. The price we have to
    pay is the appearance of constant $C$. This constant will depend
    only on the log-H{\" o}lder constant of $p$.
\end{huomautus}

\begin{proof}
    We only shortly comment on the difference in the proof.
    In the proof of \cite[Theorem 4.6]{HarHL08} between estimates
    (4.4) and (4.5), points are chosen in such a way that
    \[
        | z - y |^{ -n (p(z)-1) } < 1 + \ep
    \]
    with $y \in Y$. If strong log-H{\" o}lder continuity is relaxed
    to log-H{\" o}lder continuity, we may do the same procedures,
    but instead with
    \[
        | z - y |^{ -n (p(z)-1) } < e^C + \ep \, .
    \]
    Here $C$ depends on the log-H{\" o}lder constant of $p$. This will
    carry over to the remainder of the proof, so eventually we
    shall have
    \[
        \limsup\limits_{\delta \to 0} \varrho_{\BVp(F)}(u_\delta)
        \leq \; e^C \varrho_{\BVp(F)}(u) \, . \qedhere
    \]
\end{proof}

In \cite{HarHL08}, the authors continue to study the solutions of certain
partial differential equations. Their main result is presented as
\cite[Theorem 7.1]{HarHL08}. Roughly speaking, let us have a sequence
of variable exponent $p(\cdot)$-Laplace equations with exponents that
are bounded away from 1 and which converge to a strongly log-H{\" o}lder
continuous exponent $p$ which attains also the value 1. Then the solutions
to these equations tend to a function $u \in \BVp(\Omega)$ which is also
a solution in $\Omega \setminus Y$ and minimizes the mixed energy
$\modubvp(\cdot)$ in compact subsets of $\Omega$.

Next, we are to present a different definition for a mixed BV-Sobolev
space, which we show to be equivalent under certain assumptions. Note
that instead of studying partial differential equations, we then continue
to study capacities. This can be seen as a motivation for the new definition.
At least for us, the new definition made studying the capacity much more natural.

\section{The mixed $\BV$-Sobolev space: alternative definition}
We will now give an alternative definition for the mixed $\BV$-Sobolev space. The assumption $p^{+}<\infty$ guarantees that $\Lip(\Omega)$ is dense in $\Lp(\Omega)$, see \cite[Theorem 2.11]{KovR91}, and therefore we take the following approach based on relaxing the functional
\[u \mapsto \int_{\Omega}|\nabla u|^{p(x)}dx.\]
The density result is actually true for smooth functions as well, but we restrict our consideration to relaxation with respect to sequences of locally Lipschitz functions. As a class, Lipschitz functions have better closedness properties; especially the cases where we consider pointwise maxima and minima will be important.

\begin{maaritelma} \label{maar_BVpp_pseudomod}
    Define the pseudo-modular
    \[
        \modubvpp(u) :=
        \inf\left\{\liminf\limits_{i\to\infty}\int_{\Omega}|\nabla u_{i}(x)|^{p(x)}dx\right\},
    \]
    where the infimum is taken over all sequences $(u_{i})_{i=1}^{\infty}$
    in $\Lip(\Omega)\cap \Lp(\Omega)$ such that $u_{i}\to u$ in $\Lp(\Omega)$.
    If the basic set is some $E$ other than $\Omega$, we may emphasize this
    by writing $\varrho_{\BVpp(E)}(u)$. Define the space
    \[
        \BVpp(\Omega) := \left\{ u \in \Lp(\Omega)
                                 \; : \; \modubvpp(u) < \infty \right\} .
    \]
\end{maaritelma}

%
%

\begin{huomautus}
    For constant function $p>1$ this definition gives the ordinary Sobolev
    space $W^{1,p}(\Omega)$. Also, if $p \equiv 1$ we obtain the functions
    of bounded variation. In fact, it can be seen that
    \[
    \begin{cases} \displaystyle
        \modubvpp(u) =\int_{\Omega}|\nabla u|^{p}dx
        & \text{ for } p(\cdot) \equiv p , \; 1 < p < \infty , \\
        \modubvpp(u)=\|Du\|(\Omega)
        & \text{ for } p(\cdot) \equiv 1 .
    \end{cases}
    \]
    For these kinds of results, we provide as references
    \cite[Examples 3.13, 3.14]{Dal93} and \cite[Theorem 3.9]{AmbFP06}.
    Note that these are not the original results but rather good
    overall references.
\end{huomautus}

As a matter of definition, some elementary calculations and standard
techniques from the theory of modular spaces, we have the following.

\begin{lemma} \label{lemma_pseudomod_ominaisuudet}
    The pseudo-modular $\modubvpp$ is convex. It is continuous and decreasing
    as a mapping
    \[
        \lambda \mapsto \modubvpp\left( \frac{u}{\lambda} \right)
    \]
    for $u \in \BVpp(\Omega)$ and $\lambda > 0$.
\end{lemma}

\begin{huomautus}
    Actually, the pseudo-modular has all the properties of a continuous
    convex modular, except for the fact that
    $\modubvpp(u) = 0$ does not imply $u = 0$.
\end{huomautus}

The basic idea of Definition \ref{maar_BVpp_pseudomod} is to allow ''BV-like'' behaviour of the
functions in $Y$. It might not seem natural at first, but this gives a better
approach towards tools such as capacity than working with the Sobolev space
$\W(\Omega)$. It is also noteworthy that usually scaling with a constant $\lambda$
does not behave well in the world of variable exponent modulars. However, it can
be shown that if $\modubvpp(u)$ has a minimizing sequence $u_i$, then $\lambda u_i$
is a minimizing sequence for $\modubvpp(\lambda u)$.

We next move to defining a norm in our space. Let us define
    \[
        \|u\|_{\BVpp} := \|u\|_{p(\cdot)}
                         + \inf \left\{ \lambda >0 \; : \; \modubvpp(u/\lambda)\leq 1 \right\}
    \]
for $u \in \BVpp(\Omega)$. Let us establish that this definition yields a norm. 

\begin{teoreema} \label{normi}
    Let $\Omega\subset\Rn$ be open. Then $\BVpp(\Omega)$
    equipped with $\|\cdot\|_{\BVpp}$ is a norm space.
\end{teoreema}

\begin{proof}
    We know that $\norm{ \cdot }_{p(\cdot)}$ is a norm.
    By \cite[Theorem 1.5]{Mus83}, the convex pseudo-modular
    $\modubvpp$ defines a homogeneous pseudo-norm as
    \[
        \inf \left\{ \lambda >0 \; : \; \modubvpp(u/\lambda)\leq 1 \right\} \, .
    \]
    It is clear that the sum of a norm and homogeneous
    pseudo-norm defines a norm.
\end{proof}

One of the main motivations to consider the mixed $\BV$-Sobolev space
in this paper is the following lower semicontinuity property. A similar
property is true, and well known, in the classical BV-space, see
e.g.~\cite[Theorem 5.2.1]{Zie89}. On the other hand, such a result
is not available in the Sobolev space $W^{1,1}(\Omega)$; thus it is reasonable
to consider BV-type behaviour in $Y$.

\begin{teoreema} \label{puolijatku} Let $u_{i}\in\BVpp(\Omega)$ be such that $u_{i}\to u$ in $\Lp(\Omega)$. Then
\[\modubvpp(u)\leq\liminf\limits_{i\to\infty}\modubvpp(u_{i}).\]
\end{teoreema}

\begin{proof}
According to the definition of $\modubvpp$ for every $i=1,2,\ldots$ we can choose a function $v_{i}\in\Lip(\Omega)\cap \Lp(\Omega)$ such that
\[\|u_{i}-v_{i}\|_{p(\cdot)}<\frac{1}{i}\]
and
\[\Big|\modubvpp(u_{i})-\int_{\Omega}|\nabla v_{i}|^{p(x)}dx\Big|<\frac{1}{i}.\]
Since
\[\|u-v_{i}\|_{p(\cdot)}\leq\|u-u_{i}\|_{p(\cdot)}+\|u_{i}-v_{i}\|_{p(\cdot)}\to 0\]
as $i\to\infty$, we obtain that
\begin{align*}\modubvpp(u)&\leq\liminf\limits_{i\to\infty}\int_{\Omega}|\nabla v_{i}|^{p(x)}dx\leq\liminf\limits_{i\to\infty}\Big(\modubvpp(u_{i})+\frac{1}{i}\Big)\\
&\leq\liminf\limits_{i\to\infty}\modubvpp(u_{i}).\qedhere
\end{align*}
\end{proof}

Similarly as the space $\BVp (\Omega)$, the mixed BV-Sobolev space
$\BVpp (\Omega)$ is a Banach space.


\begin{teoreema}
The space $\BVpp(\Omega)$ equipped with the
norm $\|\cdot\|_{\BVpp}$ is a Banach space.
\end{teoreema}

\begin{proof}
If $(u_{i})_{i=1}^{\infty}$ is Cauchy sequence in $\BVpp(\Omega)$,
then it is Cauchy sequence in $\Lp(\Omega)$ and there exists $u\in\Lp(\Omega)$
such that $u_{i}\to u$ in $\Lp(\Omega)$. Every Cauchy sequence is bounded.
Boundedness in norm implies boundedness in modular in the case $p^{+} < \infty$.
Thus there exists $M>0$ such that for every $i=1,2,3\ldots$
    \[
        \modubvpp (u_i) \leq M \, .
    \]
Now Theorem \ref{puolijatku} implies  
\[\modubvpp(u)\leq\liminf\limits_{i\to\infty}\modubvpp(u_{i})<\infty\]
and $u\in\BVpp(\Omega)$. Let $\lambda>0$ and $N_{\lambda}$ be such that $i,j\geq N_{\lambda}$ implies
\[\modubvpp\Big(\frac{u_{i}-u_{j}}{\lambda}\Big)\leq 1.\]
Since $u_{i}-u_{j}\to u_{i}-u$ in $\Lp(\Omega)$ as $j\to\infty$, it follows from Theorem \ref{puolijatku} that
\[\modubvpp\Big(\frac{u_{i}-u}{\lambda}\Big)\leq\liminf\limits_{j\to\infty}\modubvpp\Big(\frac{u_{i}-u_{j}}{\lambda}\Big)\leq 1.\]
Letting $\lambda\to 0$ we have that $u_{i}\to u$ in $\BVpp(\Omega)$.
\end{proof}

In the next section we shall repeatedly use the following lemma and
Theorem \ref{puolijatku} to prove the properties of capacity.


\begin{lemma} \label{hilalemma}
    Let $u,v\in\BVpp(\Omega)$. Then
    \[\modubvpp(\max\{u,v\})+\modubvpp(\min\{u,v\})\leq\modubvpp(u)+\modubvpp(v).\]
\end{lemma}

\begin{proof}
    Let $u_{i},v_{i}\in \Lip(\Omega)\cap\Lp(\Omega)$, $i=1,2,\ldots$
    be sequences such that $u_{i}\to u$, $v_{i}\to v$ in $\Lp(\Omega)$ and
    \[
        \int_{\Omega}|\nabla u_{i}|^{p(x)}dx \to \modubvpp(u) , \;
        \int_{\Omega}|\nabla v_{i}|^{p(x)}dx\to\modubvpp(v)
    \]
    as $i\to\infty$. Clearly $\max\{u_{i},v_{i}\}\to\max\{u,v\}$
    and $\min\{u_{i},v_{i}\}\to\min\{u,v\}$ in $\Lp(\Omega)$,
    as $i\to\infty$ and hence
    \begin{align*}
        \modubvpp(\max\{u,v\}) & + \modubvpp(\min\{u,v\}) \\
        \leq & \liminf\limits_{i\to\infty}\int_{\Omega}|\nabla\max\{u_{i},v_{i}\}|^{p(x)}dx \\ 
             & + \liminf\limits_{i\to\infty}\int_{\Omega}|\nabla\min\{u_{i},v_{i}\}|^{p(x)}dx \\
        \leq & \liminf\limits_{i\to\infty}\int_{\Omega}|\nabla\max\{u_{i},v_{i}\}|^{p(x)}
               + |\nabla\min\{u_{i},v_{i}\}|^{p(x)}dx \\
        = & \liminf\limits_{i\to\infty}\int_{\Omega}|\nabla u_{i}|^{p(x)}
            + |\nabla v_{i}|^{p(x)}dx \\
        = & \lim\limits_{i\to\infty}\int_{\Omega}|\nabla u_{i}|^{p(x)}dx
            + \lim\limits_{i\to\infty}\int_{\Omega}|\nabla v_{i}|^{p(x)}dx \\
        = & \modubvpp(u)+\modubvpp(v) . \qedhere
    \end{align*}
\end{proof}

If we assume $\Omega$ to be bounded and $p$ to be
log-H{\" o}lder continuous, our definition of mixed space is
equivalent with the definition of \cite{HarHL08}.
See our Definition \ref{HarHL08_maar42}.

\begin{teoreema}
\label{ekvi}Let $\Omega\subset\Rn$ be a bounded open set and $p$ log-H{\" o}lder continuous exponent with $p^{+}<\infty$. Then $\BVpp(\Omega)=\BVp(\Omega)$.
\end{teoreema}

\begin{proof}
Let $u\in\BVp(\Omega)$ and $\Omega_{i}$ be subdomains of $\Omega$ such that  
$\overline{\Omega}_{i}\subset\Omega_{i+1}$
for every $i=1,2,\ldots$ and 
\[\bigcup\limits_{i=1}^{\infty}\Omega_{i}=\Omega.\] 
Let $U_{i}=\Omega_{i+1}\setminus\overline{\Omega}_{i-1}$ for $i=1,2\ldots$ where $\Omega_{0}=\emptyset$. 
Let $\{\psi_{i}\}_{i=1}^{\infty}$ be the partition of unity subordinate to the open cover $\{U_{i}\}_{i=1}^{\infty}$, that is functions $\psi_{i}\in C_{c}^{\infty}(U_{i})$ such that
$0\leq \psi_{i}\leq 1$ for every $i=1,2,\ldots$ and $\sum\limits_{i=1}^{\infty}\psi_{i}=1$ in $\Omega$. Let $\ep>0$. Obviously, $\overline{U}_{i}$ is a closed subset of $\Omega$ for every $i$, so according to Theorem \ref{variant_of_HarHL08_46} we can choose $\delta_{i}>0$ such that
\begin{align*}\int_{\Omega}\left|u_{\delta_{i}}\psi_{i}-u\psi_{i}\right|^{p(x)}dx&<2^{-i}\ep \\
\int_{\Omega}\left|u\nabla\psi_{i}-u_{\delta_{i}}\nabla \psi_{i}\right|^{p(x)}dx&<2^{-i}\ep \\
\varrho_{\mathrm{BV}^{p(\cdot)}(\overline{U}_{i})}(u_{\delta_{i}})&< C \varrho_{\mathrm{BV}^{p(\cdot)}(\overline{U}_{i})}(u)+2^{-i}\ep.
\end{align*}
The constant $C \geq 1$ is the constant from Theorem \ref{variant_of_HarHL08_46}.
We denote 
\begin{equation*}
v_{\ep} := \sum\limits_{i=1}^{\infty}u_{\delta_{i}}\psi_{i}.
\end{equation*}
It is clear that $v_{\ep}\in \Lip(\Omega)\cap\Lp(\Omega)$ and
\begin{align*}\int_{\Omega}\left|v_{\ep}-u\right|^{p(x)}dx&=\int_{\Omega}\left|\sum\limits_{i=1}^{\infty}u_{\delta_{i}}\psi_{i}-\sum\limits_{i=1}^{\infty}u\psi_{i}\right|^{p(x)}dx \\
&\leq\int_{\Omega}\left(\sum\limits_{i=1}^{\infty}\left|u_{\delta_{i}}\psi_{i}-u\psi_{i}\right|\right)^{p(x)}dx \\
&\leq\int_{\Omega}2^{p^{+}}\sum\limits_{i=1}^{\infty}\left|u_{\delta_{i}}\psi_{i}-u\psi_{i}\right|^{p(x)}dx \\
&\leq 2^{p^{+}}\sum\limits_{i=1}^{\infty}\int_{\Omega}\left|u_{\delta_{i}}\psi_{i}-u\psi_{i}\right|^{p(x)}dx \\
&\leq 2^{p^{+}}\ep.
\end{align*}
Thus $v_{\ep}\to u$ in $L^{p(\cdot)}(\Omega)$ as $\ep\to 0$. Since $\sum\limits_{i=1}^{\infty}\nabla \psi_{i}=0$ on $\Omega$ we obtain the following identity for the derivative
\begin{align*}\int_{\Omega}\left|\nabla v_{\ep}\right|^{p(x)}dx&=\int_{\Omega}\left|\sum\limits_{i=1}^{\infty}\nabla (u_{\delta_{i}}\psi_{i})\right|^{p(x)}dx\\
&=\int_{\Omega}\left|\sum\limits_{i=1}^{\infty}\psi_{i}\nabla u_{\delta_{i}}+u_{\delta_{i}}\nabla \psi_{i}\right|^{p(x)}dx\\
&=\int_{\Omega}\left|\sum\limits_{i=1}^{\infty}\psi_{i}\nabla u_{\delta_{i}}-\sum\limits_{i=1}^{\infty}(u\nabla \psi_{i}-u_{\delta_{i}}\nabla \psi_{i})\right|^{p(x)}dx.
\end{align*}
Thus
\begin{align*}
\int_{\Omega}&\left|\nabla v_{\ep}\right|^{p(x)}dx\\
&\leq 2^{p^{+}}\int_{\Omega}\left|\sum\limits_{i=1}^{\infty}\psi_{i}\nabla u_{\delta_{i}}\right|^{p(x)}+\left|\sum\limits_{i=1}^{\infty}(u\nabla \psi_{i}-u_{\delta_{i}}\nabla \psi_{i})\right|^{p(x)}dx\\
&\leq C\sum\limits_{i=1}^{\infty}\int_{\Omega}\left|\psi_{i}\nabla u_{\delta_{i}}\right|^{p(x)}dx+C\sum\limits_{i=1}^{\infty}\int_{\Omega}\left|u\nabla \psi_{i}-u_{\delta_{i}}\nabla \psi_{i}\right|^{p(x)}dx\\
&\leq C\sum\limits_{i=1}^{\infty}\int_{\overline{U}_{i}}\left|\nabla u_{\delta_{i}}\right|^{p(x)}dx+C\sum\limits_{i=1}^{\infty}\int_{\Omega}\left|u\nabla \psi_{i}-u_{\delta_{i}}\nabla \psi_{i}\right|^{p(x)}dx.
\end{align*}
For the last sum we have that
\[\sum\limits_{i=1}^{\infty}\int_{\Omega}\left|u\nabla \psi_{i}-u_{\delta_{i}}\nabla \psi_{i}\right|^{p(x)}dx\leq \ep.\]
The first sum can be estimated as follows
\begin{align*}
\sum\limits_{i=1}^{\infty}\int_{\overline{U}_{i}}\left|\nabla u_{\delta_{i}}\right|^{p(x)}dx&=\sum\limits_{i=1}^{\infty}\varrho_{\text{BV}^{p(\cdot)}(\overline{U}_{i})}(u_{\delta_{i}})\\
&\leq \sum\limits_{i=1}^{\infty}\big(C \varrho_{\text{BV}^{p(\cdot)}(\overline{U}_{i})}(u) + 2^{-i} \ep \big)\\
&\leq \sum\limits_{i=1}^{\infty}\left(C \|Du\|(\overline{U}_{i}\cap Y) + C \int_{\overline{U}_{i}\setminus Y}\left|\nabla u\right|^{p(x)}dx\right)+\ep\\
&\leq C \varrho_{\text{BV}^{p(\cdot)}(\Omega)}(u)+\ep.
\end{align*}
Here we used the fact that $\|Du\|(\cdot)$ is a measure and that $\sum\limits_{i=1}^{\infty}\chi_{\overline{U}_{i}}\leq 3$. Thus
\[\liminf\limits_{\ep\to 0}\int_{\Omega}|\nabla v_{\ep}|^{p(x)}dx<\infty\]
and $u\in\BVpp(\Omega)$.

Assume next that $u\in\BVpp(\Omega)$. Let $u_{i}\in \Lip(\Omega)\cap \Lp(\Omega)$, $i=1,2,\ldots$ be such that $u_{i}\to u$ in $\Lp(\Omega)$ and
\[\int_{\Omega}|\nabla u_{i}|^{p(x)}dx\to\modubvpp(u)\]
as $i\to\infty$. For $j=1,2,\ldots$ let
\[V_{j}=\Big\{x\in\Omega: p(x)>1+\frac{1}{j}\Big\}.\]
Clearly $V_{j}\subset V_{j+1}$ for every $j=1,2,\ldots$ and the sets $V_{j}$ form an open covering for $\Omega\setminus Y$. Without loss of generality we may assume that the sets $V_{j}$ are nonempty. 
The sequences $(u_{i})_{i=1}^{\infty}$ and $(\nabla u_{i})_{i=1}^{\infty}$ are bounded in $\Lp(V_{j})$ for each $j=1,2,\ldots$, and the spaces are reflexive due to $p^{-}_{V_j} > 1$. In what follows, we do a diagonalization argument: starting from index $1$, we always choose the subsequences from previous subsequences while passing from $V_j$ to $V_{j+1}$. Let now $j$ be arbitrary. By boundedness and reflexivity, there exists a subsequence, hereafter taken to be the whole sequence, which has weak limits
\[
    u_i \rightharpoonup v_j \, \text{ and } \, \nabla u_i \rightharpoonup w_j
\]
in $\Lp(V_j)$ for each $j$. Using the definition of weak derivative and uniqueness of weak limit, it can be seen that actually $w_j = \nabla v_j$. From this collection of subsequences, we pick a diagonal sequence; for simplicity we again denote this by $u_i$. The subsequence for $v_{j+1}$ has been extracted from the previous one, so we have for the diagonal sequence
    \begin{align*}
        & u_i \rightharpoonup v_j \, \text{ in } \, \Lp(V_j) , \\
        & u_i \rightharpoonup v_{j+1} \, \text{ in } \, \Lp(V_{j+1}) ,
    \end{align*}
and the same for the gradients. By the definition of weak convergence, it is an easy calculation to see that actually
    \[
        \int_{V_j} v_j w \, dx = \int_{V_j} v_{j+1} w \, dx 
    \]
for all $w \in L^{p^\prime(\cdot)}(V_j)$. Here we also use the facts that $V_j \subset V_{j+1}$ and that the dual space is just the conjugate Lebesgue space. By a well known variation lemma we now deduce that $v_j = v_{j+1}$ a.e.~in $V_j$. The same can be seen for gradients.
We can now define a function $v$ on $\Omega\setminus Y$ by setting
$v(x)=v_{j}(x)$ when $x\in V_{j}\setminus V_{j-1}$. It is clear that now
$\nabla v = \nabla v_j$ a.e.~in $V_j \setminus V_{j-1}$. Since $u_{i}\to u$ in $\Lp(V_{j})$ for every $j=1,2,\ldots$ it follows that $u=v$ a.e. in $V_{j}$ for every $j=1,2,\ldots$
Hence $u=v$ a.e. in $\Omega\setminus Y$ and $v\in\Lp(\Omega\setminus Y)$. By Lemma 2.1. in \cite{HarHL08}
\[\int_{V_{j}}|\nabla v|^{p(x)}dx\leq\liminf\limits_{i\to\infty}\int_{V_{j}}|\nabla u_{i}|^{p(x)}dx\leq\modubvpp(u)\]
for every $j=1,2,\ldots$ This implies that
\[\int_{\Omega\setminus Y}|\nabla v|^{p(x)}dx=\lim\limits_{j\to\infty} \int_{V_{j}}|\nabla v|^{p(x)}dx\leq\modubvpp(u)\]
and hence $|\nabla v|\in\Lp(\Omega\setminus Y)$. Thus $v\in\W(\Omega\setminus Y)$ and $u=v$ a.e. in $\Omega\setminus Y$ implies that $u\in\W(\Omega\setminus Y)$. Since $|\Omega|<\infty$, we obtain that $u_{i}\to u$ in $L^{1}(\Omega)$ as $i\to\infty$ and
by the lower semicontinuity of the variation measure and H{\" o}lder's inequality
    \begin{align*}
        \|Du\|(\Omega)
        \leq & \liminf\limits_{i\to\infty}\int_{\Omega}|\nabla u_{i}|dx \\
        \leq & C_{\Omega,p(\cdot)} \max\left\{ \modubvpp(u) , \left( \modubvpp(u) \right)^\frac{1}{p^{+}} \right\} \\
        < & \infty .
    \end{align*}
Thus $u\in \BV(\Omega)\cap\W(\Omega\setminus Y)$ and $u\in\BVp(\Omega)$.
\end{proof}

\begin{huomautus}
    By the above proof, it is clear that when $\Omega$ is bounded and $p$ is bounded
    and log-H{\" o}lder continuous, the two pseudo-modulars will have the following
    ''equivalence-like'' relation:
    \[
        \frac{1}{C} \modubvpp(u)
        \leq \, \modubvp(u)
        \leq C \max\left\{ \modubvpp(u) , \left( \modubvpp(u) \right)^\frac{1}{p^{+}} \right\} .
    \]
    We actually show in the next theorem that this can be improved to a true
    equivalence.
\end{huomautus}

\begin{teoreema} \label{modular_equivalence}
    Let $p$ be log-H{\" o}lder continuous, $p^{+} < \infty$, and $\Omega$ bounded. 
    We may state the following true equivalence:
    \[
        \frac{1}{C} \modubvpp(u)
        \leq \, \modubvp(u)
        \leq C \modubvpp(u) .
    \]
\end{teoreema}

\begin{proof}
    By the proof of Theorem \ref{ekvi} it is clear that
    \[
        \begin{cases}
        \varrho_{\BVpp(\Omega)}(u)
        \leq C \varrho_{\BVp(\Omega)}(u) , & \mbox{ } \\
        \varrho_{\BVp(\Omega \setminus Y)}(u)
        \leq \varrho_{\BVpp(\Omega)}(u) , & \mbox{ }
    \end{cases}
    \]
    so it remains to estimate $\norm{Du}(\Omega \cap Y)$. Let $u \in \BVpp(\Omega)$.
    The usual convention is that the variation measure is extended from open
    sets to Borel sets as
    \[
        \norm{ Du }(\Omega \cap Y)
        = \inf \left\{ \norm{ Du }(U) \; : \; U \text{ open} , \, \Omega \cap Y \subset U \right\} \, .
    \]
    Let us define the sets
    \[
        U_j := \left\{ x \in \Omega \; : \; \mathrm{dist} \left( x , \Omega \cap Y \right) < \frac{1}{j}  \right\}
    \]
    for integers $j \geq 1$ and denote
    \[
        p^{+}_j := p^{+}_{ U_j } \, .
    \]
    Since the sets $U_j$ shrink monotonously to $Y$ at the least at
    rate $\frac{1}{j}$ and $p$ is log-H{\" o}lder continuous, it is clear
    that $p^{+}_j \to 1$. Similarly as in the proof and remark of
    Theorem \ref{ekvi}, we may deduce that
    \begin{align*}
        \norm{ Du }(\Omega \cap Y)
        \leq & \norm{ Du }( U_j ) \\
        \leq & C \max \left\{ \varrho_{\BVpp(U_j)}(u) , \left( \varrho_{\BVpp(U_j)}(u) \right)^\frac{1}{p^{+}_j} \right\} \\
        \leq & C \max \left\{ \varrho_{\BVpp(\Omega)} (u) , \left( \varrho_{\BVpp(\Omega)} (u) \right)^\frac{1}{p^{+}_j} \right\}
    \end{align*}
    for all $j$. Here we have used the fact that since $u \in \BVpp(\Omega)$,
    also $u \in \BVpp(U_j)$, and we have used the estimates from the proof
    of Theorem \ref{ekvi} for $\BVpp(U_j)$ instead of $\BVpp(\Omega)$.
    Letting $j \to \infty$ and thus $p^{+}_j \to 1$ concludes the proof.
\end{proof}

\section{Mixed BV-Sobolev capacity}

In this section, we will define a mixed capacity. It is the capacity naturally induced by
the mixed space $\BVpp(\Omega)$. Let us define the admissible functions. If
$E\subset\Omega$, we denote by $\adm(E)$ functions $u\in\BVpp(\Omega)$, $0\leq u\leq 1$
such that $u=1$ in an open neighbourhood of set $E$.

\begin{maaritelma}Let $E\subset\Omega$. The mixed capacity is defined as
\[ \kapbvp(E) := \inf\left\{\modup(u)+\modubvpp(u): u\in\adm(E)\right\} . \]
\end{maaritelma}

We will study the properties of the mixed capacity and conclude that it is in
fact a Choquet capacity. In this section, we will heavily utilize the semicontinuity
property of Theorem \ref{puolijatku} and the lattice property of Lemma
\ref{hilalemma}. Similar principles are used also when dealing with the
classical BV capacity, see \cite{FedZ72} and \cite[Chapter 5.12]{Zie89}.

\begin{teoreema} \label{theor_outermeasure}
The mixed capacity $\kapbvp(\cdot)$ is an outer measure.
\end{teoreema}

\begin{proof}
Clearly $\kapbvp(\emptyset)=0$ and $E_{1}\subset E_{2}$ implies that
$\adm(E_{2})\subset\adm(E_{1})$, hence $\kapbvp(E_{1})\leq\kapbvp(E_{2})$. To prove the subadditivity we may assume that
\[\sum\limits_{i=1}^{\infty}\kapbvp(E_{i})<\infty.\]
We let $\ep>0$ and for every index $i=1,2,\ldots$ choose functions $u_{i}\in\adm(E_{i})$ such that
\[\modup(u_{i})+\modubvpp(u_{i})\leq\kapbvp(E_{i})+\ep 2^{-i}.\]
Let
\[u := \sup\limits_{1\leq i<\infty}u_{i}\]
and notice that
\[\modup(u)\leq\sum\limits_{i=1}^{\infty}\modup(u_{i})<\infty.\]
Hence $u\in\Lp(\Omega)$. We define
\[v_{j} := \max\limits_{1\leq i\leq j}u_{i}\]
and notice that $v_{j}\to u$ in $\Lp(\Omega)$ as $j\to\infty$ by dominated convergence. Therefore, by using Theorem \ref{puolijatku} and Lemma \ref{hilalemma} we obtain that
\begin{align*}\modup(u)+\modubvpp(u)&\leq\sum\limits_{i=1}^{\infty}\modup(u_{i})+\liminf\limits_{j\to\infty}\modubvpp(v_{j})\\
&\leq\sum\limits_{i=1}^{\infty}\modup(u_{i})+\liminf\limits_{j\to\infty}\sum\limits_{i=1}^{j}\modubvpp(u_{i})\\
&=\sum\limits_{i=1}^{\infty}\modup(u_{i})+\sum\limits_{i=1}^{\infty}\modubvpp(u_{i})\\
&\leq\sum\limits_{i=1}^{\infty}\kapbvp(E_{i})+\ep.
\end{align*}
Clearly $u\in\adm\big(\bigcup\limits_{i=1}^{\infty}E_{i}\big)$ and hence
\[\kapbvp\Big(\bigcup\limits_{i=1}^{\infty}E_{i}\Big)\leq\sum\limits_{i=1}^{\infty}\kapbvp(E_{i}).\qedhere\]
\end{proof}


The capacity behaves well for increasing sequence of sets. Note that this
property is not known for the variable exponent Sobolev capacity defined
in \cite{HarHKV03} in the case $p^{-} = 1$.

\begin{teoreema} \label{theor_increasingsets}
Let $E_{1}\subset E_{2}\subset\ldots\subset E_{i}\subset E_{i+1}\subset\ldots\subset\Omega$ be an increasing sequence of sets. Then
\[\kapbvp\Big(\bigcup\limits_{i=1}^{\infty}E_{i}\Big)=\lim\limits_{i\to\infty}\kapbvp(E_{i}).\]
\end{teoreema}

\begin{proof}By monotonicity
\[\lim\limits_{i\to\infty}\kapbvp(E_{i})\leq\kapbvp\Big(\bigcup\limits_{i=1}^{\infty}E_{i}\Big).\]
In order to prove the opposite inequality we may assume that
\[\lim\limits_{i\to\infty}\kapbvp(E_{i})<\infty.\]
For every index $i=1,2,\ldots$ choose function $u_{i}\in\adm(E_{i})$ such that
\[\modup(u_{i})+\modubvpp(u_{i})<\kapbvp(E_{i})+\ep 2^{-i}.\]
Let
\begin{align*}v_{i} & := \max\big\{u_{1},\ldots,u_{i}\big\} = \max\big\{v_{i-1},u_{i}\big\}\\
w_{i} & := \min\big\{v_{i-1},u_{i}\big\},
\end{align*}
and notice that $v_{i},w_{i}\in\BVpp(\Omega)$ for every index $i=1,2,\ldots$ and
\[E_{i-1}\subset\text{int}\big\{w_{i}\geq 1\big\}.\]
We define $E_{0}=\emptyset$ and $v_{0}\equiv 0$ and by using Lemma \ref{hilalemma} we obtain
    \begin{align*}
        \modup(v_{i}) + & \modubvpp(v_{i}) + \kapbvp(E_{i-1}) \\
        \leq & \vphantom{\int_\Omega} \;\; \modup(v_{i}) \; + \; \modubvpp(v_{i}) \;
                    + \; \modup(w_{i}) \; + \; \modubvpp(w_{i}) \\
        = & \;\; \modup(\max\{v_{i-1},u_{i}\}) \;
                 + \; \modubvpp(\max\{v_{i-1},u_{i}\}) \\
             & \vphantom{\int_\Omega} + \; \modup(\min\{v_{i-1},u_{i}\}) \;
                 + \; \modubvpp(\min\{v_{i-1},u_{i}\}) \\
        \leq &  \;\; \modup(v_{i-1}) \; + \; \modubvpp(v_{i-1}) \;
                     + \; \modup(u_{i}) \; + \; \modubvpp(u_{i}) \\
        \leq & \vphantom{\int_\Omega} \;\; \modup(v_{i-1}) \; + \; \modubvpp(v_{i-1}) \;
                    + \; \kapbvp(E_{i}) \; + \; \ep2^{-i} .
    \end{align*}
Thus, by adding these inequalities consecutively up to index $i$, we see
a telescope sum and obtain
\[\modup(v_{i})+\modubvpp(v_{i})\leq\kapbvp(E_{i})+\sum\limits_{j=1}^{i}\ep2^{-j}.\]
We define a function
\[v := \lim\limits_{i\to\infty}v_{i}\]
and by monotone convergence we obtain that
\[ \modup(v) = \lim\limits_{i\to\infty}\modup(v_{i})\leq\lim\limits_{i\to\infty}\kapbvp(E_{i}) + \ep . \]
We note that $v_i \to v$ in $\Lp(\Omega)$ by dominated convergence,
and by using Theorem \ref{puolijatku} we have that
\[\modubvpp(v)\leq\liminf\limits_{i\to\infty}\modubvpp(v_{i})\leq\lim\limits_{i\to\infty}\kapbvp(E_{i})+\ep.\]
Thus $v \in \adm \left( \bigcup\limits_{i=1}^{\infty}E_{i} \right)$, and we have the estimate
    \begin{align*}
        \kapbvp \left( \bigcup\limits_{i=1}^{\infty}E_{i} \right)
        & \leq \modup(v) + \modubvpp(v) \\
        & \leq \liminf\limits_{i\to\infty} \modup(v_{i})
          + \liminf\limits_{i\to\infty} \modubvpp(v_{i}) \\
        & \leq \liminf\limits_{i\to\infty} \left( \modup(v_{i}) + \modubvpp(v_{i}) \right) \\
        & \leq \lim\limits_{i\to\infty} \kapbvp(E_{i}) + \ep .
    \end{align*}
Letting $\ep \to 0$ completes the proof.
\end{proof}

The following theorem states that $\kapbvp(\cdot)$ is an outer capacity.
This theorem actually does not depend on the properties in Section 4.
This is a general property of almost all capacities that are defined in a
similar way; the important point in the definition is $u \equiv 1$
in an open neighbourhood.

\begin{teoreema} \label{ulkokap}
    For any $E\subset \Omega$ we have
    \[
       \kapbvp(E) =
       \inf \left\{ \kapbvp(U) \; : \; E \subset U \subset \Omega , \; U \text{ an open set } \right\} .
    \]
\end{teoreema}

\begin{proof}
    By monotonicity
    \[
        \kapbvp(E) \leq
        \inf \left\{ \kapbvp(U) \; : \; E \subset U \subset \Omega , \; U \text{ an open set } \right\} .
    \]
We can assume that $\kapbvp(E)<\infty$. Let $\ep>0$ and take $u\in\adm(E)$ such that
\[\modup(u)+\modubvpp(u)<\kapbvp(E)+\ep.\]
Since $u\in\adm(E)$ there is an open set $U$, $E\subset U\subset \Omega$ such that $u=1$ on $U$, which implies
\[\kapbvp(U)\leq\modup(u)+\modubvpp(u)<\kapbvp(E)+\ep.\]
Hence
    \[
        \inf \left\{ \kapbvp(U) \; : \; E \subset U \subset \Omega , \; U \text{ an open set } \right\}
        \leq \kapbvp(E) . \qedhere
    \]
\end{proof}

The capacity behaves well for decreasing sequence of compact sets.
This is also a general property not depending on tools from Section 4.

\begin{teoreema} \label{laskekap} If $K_{1}\supset\ldots \supset K_{i}\supset K_{i+1}\supset\ldots$ are compact subsets of $\Omega$ and $K=\bigcap\limits_{i=1}^{\infty}K_{i}$, then
\[\kapbvp(K)=\lim\limits_{i\to\infty}\kapbvp(K_{i}).\]
\end{teoreema}

\begin{proof} By monotonicity
\[\lim\limits_{i\to\infty}\kapbvp(K_{i})\geq\kapbvp(K).\]
Let $U\subset\Omega$ be an open set containing $K$. Now by the compactness of $K$, $K_{i}\subset U$ for all sufficiently large $i$. 
Therefore
\[\lim\limits_{i\to\infty}\kapbvp(K_{i})\leq\kapbvp(U)\]
and since $\kapbvp(\cdot )$ is an outer capacity, see Theorem \ref{ulkokap}, we obtain the claim by taking infimum over all open sets $U$ containing $K$.
\end{proof}

The mixed capacity satisfies the following strong subadditivity property. 

\begin{teoreema} \label{vahsub} If $E_{1},E_{2}\subset \Omega$, then 
\[\kapbvp(E_{1}\cup E_{2})+\kapbvp(E_{1}\cap E_{2})\leq \kapbvp(E_{1})+\kapbvp(E_{2}).\]
\end{teoreema}

\begin{proof}
We can assume that $\kapbvp(E_{1})+\kapbvp(E_{2})<\infty$. Let $\ep>0$ and $u_{1}\in\adm(E_{1})$ and $u_{2}\in\adm(E_{2})$ be such that
\begin{align*}
\modup(u_1)+\modubvpp(u_1)&<\kapbvp(E_{1})+\frac{\ep}{2},\\
\modup(u_2)+\modubvpp(u_2)&<\kapbvp(E_2)+\frac{\ep}{2}.
\end{align*}
We see that
    \begin{align*}
        & \max\{u_1,u_2\} \in \adm(E_1\cup E_2) , \\
        & \min\{u_1,u_2\}\in\adm(E_1\cap E_2) .
    \end{align*}
Therefore, by Theorem \ref{hilalemma}, we obtain
    \begin{align*}
        \kapbvp(E_1&\cup E_2) + \kapbvp(E_1\cap E_2) \\
        \leq & \;\; \vphantom{\int_\Omega} \modup(\max\{u_1,u_2\}) \;
                                      + \; \modup(\min\{u_1,u_2\}) \\
                                      & + \modubvpp(\max\{u_1,u_2\}) \;
                                        + \; \modubvpp(\min\{u_1,u_2\}) \\
        \leq & \;\; \vphantom{\int_\Omega} \modup(u_1) \; + \; \modubvpp(u_1) \;
                                           + \; \modup(u_2) \; + \; \modubvpp(u_2) \\
        \leq & \;\; \kapbvp(E_1) \; + \; \kapbvp(E_2) \; + \; \ep .
    \end{align*}
Letting $\ep\to 0$, we obtain the claim.
\end{proof}

By Theorems \ref{theor_increasingsets}, \ref{ulkokap}, and \ref{laskekap}, the
mixed capacity is a Choquet capacity. An important feature is that now the
capacity of a Borel set $E$ can be estimated ''from the inside'' by a compact
set, and ''from the outside'' by an open set:
    \begin{align*}
       \kapbvp(E)
       = \; & \sup \left\{ \kapbvp(K) \; : \; K \subset E , \; K \text{ compact} \right\} \\
       = \; & \inf \left\{ \kapbvp(U) \; : \; E \subset U , \; U \text{ open} \right\} .
    \end{align*}
For the original paper on abstract capacity by Choquet,
see \cite{Cho59}.

\section{Mixed capacity and Sobolev capacity}

In this section we study the relations between $\BVpp$-capacity defined
in the previous section and the $p(\cdot)$-Sobolev capacity, see
\cite{HarHKV03}. 
Let $E\subset\Rn$ and denote 
    \[
        S_{p(\cdot)}(E) :=
        \left\{ u\in\W(\Rn) \; : \; u \geq 1 \text{ in an open set containing } E \right\} .
    \]
The $p(\cdot)$-capacity of $E\subset\Rn$ is defined as
    \[
        \kapp(E) := \inf \int_{\Rn} |u|^{p(x)} + | \nabla u |^{p(x)} \, dx ,
    \]
in other words
    \[
        \kapp(E) = \inf \moduyp(u) .
    \]
The infimum is taken over functions $u\in S_{p(\cdot)}(E)$. It is easy
to see that restricting the admissible functions to the case
$0 \leq u \leq 1$ yields the same infimum. In this case, it is obviously
possible to drop the absolute value from $|u|$.

The capacity $\kapp(\cdot)$
has many favourable properties: it is monotone, it is an outer capacity,
it is finitely strongly subadditive, it has the compact set intersection property,
it is subadditive for null sets. See \cite[Theorem 3.1, Lemma 3.5]{HarHKV03}.
The Sobolev capacity has some open issues as well. In the case
$p^{-} = 1$, it is not known whether the Sobolev capacity is in general
subadditive or whether it has the increasing set union property.
Note that we have been able to solve these for the mixed capacity
in Theorems \ref{theor_outermeasure}, \ref{theor_increasingsets}.

The next theorem shows that $\kapp(\cdot)$ and $\kapbvp(\cdot)$ have the same null sets, if the variable exponent $p(\cdot)$ is bounded and log-H{\" o}lder continuous.

%
%

\begin{teoreema}
Let $p$ be a log-H{\" o}lder continuous variable exponent with $p^{+}<\infty$ and let $E\subset \Rn$. Then $\kapp(E)=0$ if and only if
$\kapbvp(E)=0$.
\end{teoreema}

\begin{proof}
Since $\kapbvp(\cdot)$ is an outer measure and $\kapp(\cdot)$ is monotone and subadditive for null sets, we may assume that $E$ is bounded. Assume first that
$\kapbvp(E)=0$. Let $0 < \ep < 1$ and 
take $u\in \adm(E)$ such that
\[ \int_{\Rn}u^{p(x)}dx+\varrho_{\BVpp(\Rn)}(u) < \ep^{p^{+}} . \]
Denote $B_{\rho}=B(0,\rho)$ and let $\rho>0$ be such that $E\subset B_{\rho}$. 
Let $\eta$ be a Lipschitz function such that $0\leq \eta\leq 1$,
$|\nabla \eta| \leq 1$, $\eta=1$ in $B_{\rho}$ and $\eta=0$ in $\Rn\setminus B_{\rho+1}$.
Denote $v := \eta u$. Now $v\in\BVpp(B_{\rho+2})$ and
\[ \int_{B_{\rho+2}}v^{p(x)}dx+\varrho_{\BVpp(B_{\rho+2})}(v)< C \ep^{p^{+}} . \]
We may now assume that $C \ep^{p^{+}} < 1$ by choosing a new, smaller $\ep$ if necessary.
Theorem \ref{ekvi} together with its proof and and H{\" o}lder's inequality implies
that $v\in\BVp(B_{\rho+2})$ and
\[ \int_{B_{\rho+2}}v\,dx+\|Dv\|(B_{\rho+2}) \leq C \ep + C \ep \leq C \ep , \]
where the constant $C$ depends only on $\rho$ and $p^{+}$. From the
Cavalieri principle \cite[Lemma 1.5.1]{Zie89} and coarea formula for
BV-functions \cite[Theorem 5.4.4]{Zie89} we deduce that
    \[
        \left| \left\{ v>t_{0} \right\} \cap B_{\rho+2} \right|
        + P( \left\{ v>t_{0} \right \} , B_{\rho+2} )
        \leq \int_{B_{\rho+2}} v \, dx + \|Dv\|(B_{\rho+2}) ,
    \]
for some $0<t_{0}<1$. Denote $E_{t_{0}} := \{v>t_{0}\}\cap B_{\rho+2}$. Since $E_{t_{0}}\subset\subset B_{\rho+2}$ we have that $P(E_{t_{0}},B_{\rho+2})=P(E_{t_{0}},\Rn)$;
see \cite[Remark 1.7]{Giu84}. Denote the set of Lebesgue
density points
    \[
        E_{t_0}^\ast
        := \left\{ x \in E_{t_0} \; : \; \lim\limits_{r \to 0} \frac{ |B(x,r) \cap E_{t_0}| }{ |B(x,r)| } = 1 \right\} .
    \]
It is known that almost every point of a measurable set is a Lebesgue
density point. We apply the modified Boxing inequality
\cite[Lemma 4.2]{HakK10} for the radius $R=1$ to obtain a covering
for $E_{t_{0}}^\ast$. We shall have
    \[ E_{t_{0}}^\ast \subset \bigcup\limits_{i=1}^{\infty}B(x_{i},5r_{i}) \]
and
    \[
        \sum\limits_{i\in I_{1}} | B(x_{i},5 r_{i}) |
        + \sum\limits_{i\in I_{2}}\frac{ |B(x_{i},5 r_{i}) | }{ 5 r_{i} }
        \leq C \left( |E_{t_0}| + P(E_{t_0},\Rn) \right).\]

Denote by $I_{3}$ indices $i\in I_{1}\cup I_{2}$ such that
$p_{10 B_{i}}^{-}>1$ where $B_{i}=B(x_{i},r_{i})$. If we have some
ball $B$, we futher denote by $\kappa B$ the cocentric ball
with the radius of the original ball $B$ scaled by some
constant $\kappa > 0$. For indices $i\in (I_{1}\cup I_{2})\setminus I_{3}$
define
\[\varphi_{i}(x) := \left( 1-\frac{\text{dist}(x,B(x_{i},5r_{i}))}{5 r_{i}} \right)^{+}.\]
For indices $i\in I_{3}$ we choose function $u_{i}\in \adm(E)$ such that
\[\frac{1}{r_{i}^{p^{+}}}\int_{B(x_{i},10r_{i})}u_{i}^{p(x)}+|\nabla u_{i}|^{p(x)}dx<2^{-i}\ep,\]
here the assumption $p_{10 B_{i}}^{-}>1$ ensures that $u_{i}\in\W(10B_{i})$. 
For $i\in I_{3}$ let
\[\varphi_{i}:=\eta_{i}u_{i},\]
where $\eta_{i}$ is $1/5r_{i}$-Lipschitz function such that $0\leq \eta_{i}\leq 1$, $\eta_{i}=1$ in $B(x_{i},5r_{i})$ and $\eta_{i}=0$ in $\Rn\setminus B(x_{i},10r_{i})$.
 
We have now defined $\varphi_i$ for all indices $i$. Let
\[\varphi := \sup\limits_{1\leq i<\infty}\varphi_{i}\]
and
\[g_{\varphi} := \sup\limits_{1\leq i<\infty}|\nabla\varphi_{i}|.\]
%
%
We now have
    \begin{align*}
        \int_{\Rn} g_{\varphi}^{p(x)} dx
        & \leq \sum\limits_{i=1}^{\infty}
               \int_{10B_{i}} |\nabla\varphi_{i}|^{p(x)} dx \\
        & = \sum\limits_{i\in (I_{1}\cup I_{2})\setminus I_{3}}
            \int_{10B_{i}} |\nabla\varphi_{i}|^{p(x)} dx
            + \sum\limits_{i\in I_{3}}
              \int_{10B_{i}} |\nabla\varphi_{i}|^{p(x)} dx .
    \end{align*}
Since $p_{10 B_{i}}^{-}=1$ for every $i\in(I_{1}\cup I_{2})\setminus I_{3}$,
we write 
   %
   %
    \begin{align*}
        \sum\limits_{i\in (I_{1}\cup I_{2})\setminus I_{3}} 
        \int_{10B_{i}}|\nabla\varphi_{i}|^{p(x)}dx
        & \leq \sum\limits_{i\in (I_{1}\cup I_{2})\setminus I_{3}}
               \int_{10B_{i}} \left(\frac{1}{5r_{i}} \right)^{p(x)} dx \\
        & = \sum\limits_{i\in (I_{1}\cup I_{2})\setminus I_{3}}
            \int_{10B_{i}} \left( 5 r_i \right)^{ - (p(x)-1) }
                           \frac{1}{5r_{i}} dx \\
        & \leq C \sum\limits_{i\in (I_{1}\cup I_{2})\setminus I_{3}}
                 \int_{10B_{i}}\frac{1}{5r_{i}} dx .
    \end{align*}
Here we used the property of log-H{\" o}lder continuous exponent $p$:
    \[
        r^{ - (p(x)-1) } \leq C .
    \]
Note that always $p(x) \leq p^{+}_{10B_i}$ and now $p^{-}_{10B_i} = 1$.
We then see that
    \begin{align*}
        C \sum\limits_{i\in (I_{1}\cup I_{2})\setminus I_{3}}
          \int_{10B_{i}}\frac{1}{5r_{i}} dx
        & \leq C \sum\limits_{i\in (I_{1}\cup I_{2})\setminus I_{3}}
                 \frac{ |5B_{i}| }{5r_{i}} \\
        & \leq C \sum\limits_{i\in I_{1}\setminus I_{3}}
                 \frac{ |5B_{i}| }{5r_{i}}
               + C \sum\limits_{i\in I_{2}\setminus I_{3}}
                   \frac{ |5B_{i}| }{5r_{i}} \\
        & \leq C \sum\limits_{i\in I_{1}\setminus I_{3}} |5B_{i}|
               + C \sum\limits_{i\in I_{2}\setminus I_{3}}\frac{ |5B_{i}| }{5r_{i}}\\
        & \leq C \left( |E_{t_0}| + P(E_{t_0},\Rn) \right) \\
        & < \; C\ep.
\end{align*}
Here we used the fact that $r_{i}\geq 1/2$ for $i\in I_{1}\setminus I_{3}$. The constant $C$ depends only on $\rho$, $p^{+}$, $n$ and the constant in the log-H{\" o}lder continuity condition. On the other hand, for $I_{3}$ we have
\begin{align*}\sum\limits_{i\in I_{3}}\int_{10B_{i}}|\nabla\varphi_{i}|^{p(x)}dx&\leq 2^{p^{+}}\sum\limits_{i\in I_{3}}\int_{10B_{i}}|\nabla u_{i}|^{p(x)}+\Big(\frac{u_{i}}{5r_{i}}\Big)^{p(x)}dx\\
&\leq C\sum\limits_{i\in I_{3}}\frac{1}{r_{i}^{p^{+}}}\int_{10B_{i}}u_{i}^{p(x)}+|\nabla u_{i}|^{p(x)}dx\\
&\leq C\sum\limits_{i\in I_{3}}2^{-i}\ep<C\ep.
\end{align*}
Thus
\[ \int_{\Rn}g_{\varphi}^{p(x)}dx<C\ep . \]
%
%
Now $\varphi$ is compactly supported bounded function, so $\varphi \in L^1(\R^n) \cap L^\infty(\R^n)$.
By H{\" o}lder's inequality, also $g_\varphi \in  L^1(\R^n)$. By
\cite[Chapter 4.7.1, Lemma 2, iii.]{EvaG92} $\varphi$ now has a weak
gradient in $L^1(\R^n)$ and $|\nabla \varphi| \leq g_\varphi$ almost everywhere.
Thus $\varphi\in W^{1,p(\cdot)}(\Rn)$ with
    \[ \int_{\Rn} |\nabla\varphi|^{p(x)} dx < C \ep . \]
Furthermore
    \begin{align*}
        \int_{\Rn} \varphi^{p(x)} dx 
        & \leq \sum\limits_{i=1}^{\infty}
               \int_{10B_{i}}\varphi_{i}^{p(x)}dx \\
        & \leq \sum\limits_{i\in (I_{1}\cup I_{2})\setminus I_{3}}
               \int_{10B_{i}} \varphi_{i}^{p(x)} dx
               + \sum\limits_{i\in I_{3}}\int_{10B_{i}}
                 \varphi_{i}^{p(x)} dx \\
        & \leq C \sum\limits_{i\in (I_{1}\cup I_{2})\setminus I_{3}}
                 \frac{ |5B_{i}| }{5r_{i}}
                 + \sum\limits_{i\in I_{3}}
                   \int_{10B_{i}} u_{i}^{p(x)} dx .
        \end{align*}
Arguing similarly as earlier, we have that
\[ \sum\limits_{i\in (I_{1}\cup I_{2})\setminus I_{3}} \frac{ |5B_{i}| }{5r_{i}} < C \ep . \]
This combined with the fact that
\[\sum\limits_{i\in I_{3}}\int_{10B_{i}}u_{i}^{p(x)}dx<\ep\]
implies
\[\int_{\Rn}\varphi^{p(x)}dx<C\ep.\]
Clearly $E\subset\text{int}\{u = 1\}\cap B_{\rho}\subset E_{t_{0}}^{*}\subset\bigcup\limits_{i=1}^{\infty}B(x_{i},5r_{i})$. For indices $i\in(I_{1}\cup I_{2})\setminus I_{3}$ we have that
$\varphi=1$ in $5B_{i}$. On the other hand, for indices $i\in I_{3}$
we have that
\begin{align*}E\cap 5B_{i}&\subset\text{int}\{u_{i}=1\}\cap 5B_{i}\\
    &=\text{int}\{\varphi_{i}=1\}\cap 5B_{i}\\
    &\subset\text{int}\{\varphi=1\}\cap 5B_{i} .
\end{align*} 
Thus $E\subset\text{int}\{\varphi=1\}$ and 
\[\kapp(E)\leq\int_{\Rn}\varphi^{p(x)}+|\nabla \varphi|^{p(x)}dx<C\ep.\]
The claim now follows by letting $\ep\to 0$.

Assume then that $\kapp(E) = 0$. This converse proof is much
simpler. Since the smooth functions are dense in $\W(\Rn)$ we
have that $S_{p(\cdot)}(E)\subset\adm(E)$. By \cite[Lemma 2.6]{HarHKV03},
    \[
        \modubvpp (u) \leq \modup (\nabla u)
    \]
for Sobolev functions. Thus $\kapbvp(E)\leq \kapp(E)$ and $\kapbvp(E)=0$.
\end{proof}

In the above proof, we made use of \cite[Chapter 4.7.1, Lemma 2, iii.]{EvaG92}
when showing that the function $\varphi$ is a Sobolev function. We would
like to mention that the same result can be shown using 1-weak upper gradients,
which are familiar to people working on analysis on metric measure spaces.
One might argue that this route is more direct and shorter. However, in that
case some additional care needs to be taken when distinguishing between functions
and their precise representatives, between the Sobolev space and the Newtonian
space. See \cite{BjoB_pp09,Sha00}.

\vspace*{36pt}
\textbf{Acknowledgements.} Acknowledgements, affiliations, and funding information are to
be added later.

\textbf{Keywords:} capacity, functions of bounded variation, Sobolev spaces, variable exponent

\textbf{Mathematics subject classification 2000:} 46E35, 26A45, 28A12


\end{document}